\documentclass[12pt,a4paper]{article} 
\usepackage{amsmath,amssymb,amsthm,amsfonts,amscd,euscript,verbatim, t1enc, newlfont, graphicx, cancel}
\usepackage{hyperref}
\usepackage[all]{xy}

\theoremstyle{definition}
\newtheorem{theo}{Theorem}[subsection]

\newtheorem{pr}[theo]{Proposition}

 \newtheorem{lem}[theo]{Lemma}

\theoremstyle{remark}
\newtheorem{rema}[theo]{Remark}

\theoremstyle{definition}
\newtheorem{defi}[theo]{Definition}

\numberwithin{equation}{subsection}

\newcommand\goz{\mathfrak o_0}
\newcommand\go{\mathfrak o}
\newcommand\gok{\mathfrak O}

\newcommand\ga{\mathfrak A}
\newcommand\gaz{\ga^0}

\newcommand\gc{\mathfrak C}
\newcommand\gd{\mathfrak D}

\newcommand\gm{\mathfrak M}
\newcommand\gmo{\gm_{\go}}

\newcommand\cpi{c_{\pi}}
\newcommand\rcpi{r(c_{\pi})\ob}
\newcommand\rx{r^{\kkk}}
\newcommand\kp{K'/k'}
\newcommand\rxp{r^{\kkp}}
\newcommand\kkp{K'\otimes_{k'}K'}
\newcommand\pip{\pi'}
\newcommand\gmp{\gm'}
\newcommand\rr{R_{\kk}}
\newcommand\rrp{R_{\kp}}
\newcommand\ovkp{\overline{k'}}
\newcommand\ovk{\overline{k}}
\newcommand\ovh{\overline{h}}
\newcommand\gop{{\mathfrak O}_{K'}}
\newcommand\gopo{\go'}
 \newcommand\wh{{\tilde{h_2}}}

\newcommand\si{\sigma}
\newcommand\osi{\overline{\si}}

\newcommand\al{\alpha}
\newcommand\eps{\varepsilon}
\newcommand\be{\beta}
\newcommand\de{\delta}

\newcommand\tr{\operatorname{tr}}

\newcommand\cha{\operatorname{char}}

\newcommand\ot{\otimes}
\newcommand\gk{{\mathfrak O}_K}
\newcommand\ob{^{-1}}
\newcommand\kk{{K/k}}
\newcommand\kkk{{K\ot_k K}}
\newcommand\kkkp{{K'\ot_k K'}}
\newcommand\gp{\rho} 

\newcommand\kg{{K[G]}}
\newcommand\kog{{k[G]}}
\newcommand\kzg{{k_0[G]}}
\newcommand\kz{{k_0}}
\newcommand\koz{k'_0}
\newcommand\vp{v_p}

\newcommand\q{\mathbb{Q}}
\newcommand\z{\mathbb{Z}}

\newcommand\mi{\underset F -}
\newcommand\pik{\pi_k}
\newcommand\piz{\pi_0}

\newcommand\zpz{\mathbb{Z}/p\z}

\newcommand\gal{\operatorname{Gal}}

 \newcommand\lan{\langle}
\newcommand\ra{\rangle}

\newcommand\ns{\{0\}}

\begin{document}


\title{
Bases of associated Galois modules in general wildly ramified extensions and in elementary abelian extensions of degree $p^2$}
\author{Mikhail V.\,Bondarko, Kirill S.\,Ladny, Konstantin I. Pimenov
\thanks{The  research of Mikhail V.\,Bondarko was supported by the Leader (Leading scientist Math) grant no. 22-7-1-13-1 of  the Theoretical Physics and Mathematics Advancement Foundation «BASIS». The research of Kirill S. Ladny was supported within the framework of the project “International academic cooperation” HSE University. } 
  }\maketitle
 \vspace*{15mm}
\begin{flushright}
	\textit{ 
 The paper is dedicated to  Sergei Vladimirovich Vostokov,\\ who was an outstanding mathematician and also \\ a wonderful scientific advisor of the first author\\
  that inspired him to study additive Galois modules.\\ Your memories will live on long after you've passed.}\end{flushright} 


\begin{abstract} 
For a wildly ramified extension $\kk$ of complete discrete valuation fields we study collections of elements of $k[G]$ (where $G=\gal(\kk)$) that fit well for constructing bases of various associated Galois modules and orders. In the case $G=(\zpz)^2$ (where $p$ is the characteristic of residue fields) we are able to compute the action of the elements $(\si_1-1)^i(\si_2-1)^j,\ 0\le i,j\le p-1,$ on the valuation filtration; here $\si_1,\si_2$ are generators of $G$. If the ramification jumps of $\kk$ are distinct modulo $p^2$ then these elements do yield "good enough"  bases in question. 
\end{abstract}


\tableofcontents

\section{Introduction}

\subsection{Some history and motivation}

This paper is devoted to the calculation of certain associated Galois modules. Those are closely related to associated Galois orders (and we compute some of those as well). Let us recall some history and motivation for studying these matters.

Starting from  \cite{L}, the ring of integers $\gok$ of a global or a local field $K$ is
studied as a module over its associated order $$\ga(\gok)=\{f\in k[G]:\ f(\gk)\subset \gk \};$$ here $G$ is the Galois group of an 
 extension $\kk$.
The main results of ibid. says that $\gok$ is free over $\ga(\gok)$ (and hence isomorphic to $\ga(\gok)$ as a Galois module) whenever $k=\q$ and $G$ is abelian.

Since 1959, dozens of papers computing associated orders (of rings of integers and fractional ideals) were written. Yet the main subject of this paper are so-called associated Galois modules that are somewhat distinct from associated orders, even though closely related to them. 
First we define them, and then discuss the relation to associated orders.

Throughout this paper $\kk$ is a totally ramified extension of complete discrete valuation fields of degree $n$. 

We set  $$\gc_i=\{f\in K[G]:\
\min_{x\in K^*}(v(f(x))-v(x))\ge i\};\ \ga_i=\gc_i\cap \kog,$$
where $v$ is the discrete valuation on $K$ and  $i\in \z$. The main question we studied is to find a specific description of all $\ga_i$.
Our main result is Theorem \ref{tmain} below that gives a simple expression for $\ga_i$ for all $i\in \z$ assuming that $\kk$ is an elementary abelian extension of degree $n=p^2$ whose ramification jumps are distinct modulo $n$.

\begin{rema}\label{rord}
1. Obviously $\ga_0\subset \ga(\gok)\subset \ga_{1-n}$ and $\ga_{1-n}\subset \ga_{-n}=\pik\ob \ga_0$, where $\pik$ is an uniformizing element of $k$.

Hence the computation of $\ga_i$ does give some information on the associated order.

2. In certain cases we have  Galois module isomorphisms $\ga_i\cong \gm^{i-d}$ (where $d$ is the ramification depth on $\kk$; see \S\ref{sbd} below); see Theorem 4.3.2(5) of \cite{bgm3}. However, 
 we probably don't have any isomorphisms of this sort in the setting of our Theorem \ref{tmain}; cf. Theorem 3.10 of \cite{by}.

  So, to add some motivation, we will now recall an interesting statement that relates associated Galois modules to the arithmetic of the extension.

  We will never apply this theorem in the paper; respectively, the reader who does not need additional motivation for studying associated Galois modules may skip it.

\end{rema}

Let $F$ be a commutative $m$-dimensional formal group law with coefficients in $\go$. We will write  $F(\gm)$ for the corresponding value of the commutative group functor coming from $F$ (here $\gm$ is the maximal ideal of the valuation field $K$); note that there exists a canonical (forgetful) bijection of sets $F(\gm)\to \gm^m$.

\begin{theo}[Theorem 2.3.1  of \cite{bgm4}]

Take $w>0$; let a map 
 $$ A:G\to F(\gm),\
\si\to a_\si= (a_{1\si},\dots ,a_{m\si})$$ (here $a_{i\si}\in \gm$) belong to $Z^1(G, F(\gm))$ (that is, an inhomogeneous Galois $1$-cocycle).  Then the following
statements are equivalent.

1. For any $i,\ 1\le i\le m$, the element $f_i=\sum_{\si\in
G}a_{i\si}\si$ lies in the module $\gc_{w+d}$; here $d$ is the  ramification depth of $\kk$ (see \S\ref{sbd} below). 

2. $f\in B^1(G,F(\gm))$ (a $1$-coboundary) and there exists an $x=(x_i)\in
F(\gm)$ such that $$
x{\mi} \si(x)
  =a_\si\ \forall\si\in G $$
and $v(x_i)\ge w$ for all $i$.  
\end{theo}

\begin{rema}\label{rint}
1. We are mainly interested in the case where the values of $A$ 
  belong to  $F(\gmo)=F(\gm\cap k)$. Then  $A$ is a homomorphism $G\to F(\gmo)$. 
Moreover, in this case condition 1 of the theorem depends on the modules $\ga_i$ only and condition 2 translates into a Kummer-type condition on the corresponding formal modules; see Definition 4.1.1 and Theorem 4.2 of \cite{bgm4}.  

2. Moreover, the theorem demonstrates that in the case where 
 all  $a_{i\si}$ belong to a subfield $k_0$ of $k$ it suffices to know the modules  $\gaz_i=\ga_i\cap k_0[G]$ only to apply the theorem. For this reason (and also to extend Theorem \ref{tmaini} below to more general extensions) we develop some theory of associated Galois modules of this type. 
 \end{rema}
 
 \subsection{The contents of the paper}
 
 The main specific result of the paper is the following one.
 
 \begin{theo}\label{tmaini}
 Assume $\kk$ is a totally ramified Galois extension with Galois group $G\cong (\zpz)^2$, with ramification jumps that are 
  distinct modulo $p^2$, and $\si_1,\si_2$ are elements of $G$ corresponding to distinct ramification jumps.
 
 Then there exists a piecewise linear function $H:\z^2\to \z$ such that for any $l\in \z$ we have $\ga_l=\lan \pik^{[(l-d-H(i,j)-1)/n]+1} (\si_1-1)^i(\si_2-1)^j|\ 0\le i,j\le p-1\ra_{\go}$. 
\end{theo}
\begin{proof} Immediate from Theorem \ref{tmain}(2) combined with  Proposition \ref{pbase}(2).
\end{proof}

\begin{rema}\label{rmaini}
1. The function $H$ is  piecewise linear; it is calculated in \S\ref{smain} as well.

2. It suffices to assume that  
 $\si_1$ and $\si_2$ generate $G$  
 to compute the smallest $l$ such that  $(\si_1-1)^i(\si_2-1)^j\in \ga_l$ for each $(i,j):\ 0\le i,j\le p-1$; see  Theorem \ref{tmain}(1) and Remark \ref{rmain}(3).


\end{rema}

 Now, let us say something about the general results and methods of this paper.

We  develop much general theory for specifying ``nice''  bases of various associated Galois modules (and orders). Whereas most of this theory is inspired by the examples considered in \cite{bgm3}, our definitions and statements related to these {\it graded-independent} and {\it diagonal} bases appear 
to be  completely new. 
 
 However, some of the statements and methods that we use for studying these bases have much in common with  \cite{bgm3}. In particular, some of the calculations are based on  the isomorphism $\phi:\kkk\to \kg$ introduced in \cite{bgm2}, and we prove some new statements related to it.

 In \S\ref{sbase} we introduce some notation  and recall some theory of associated Galois modules. Most of the statements in this section are related to  $\phi$ and $\kkk$.

 In \S\ref{scalc} we develop the theory of graded bases; they allow us to construct 
 "compatible"  bases for all $\ga_i$. We also make some general calculations and prove Theorem \ref{tmaini}.
 
 In \S\ref{smore} we define and study relative Galois modules of the sort $\gaz_i$ (see Remark \ref{rint}(2)). They allow us to extend Theorem \ref{tmaini}; see Theorem \ref{trel}(2).
 
 Next we 
  define some more associated Galois modules (and orders; we define $\gaz(i,j)=\{f\in \kz[G]:\ f(\gm^i)\subset \gm^j \}$). Next we study those (graded) bases that are convenient for constructing bases of modules of this type; we call them {\it diagonal bases}. We describe an algorithm that allows to describe bases of all $\gaz(i,j)$ in terms of a diagonal base, and also prove that certain tame lifts of Galois extensions contain graded bases. 
  
  Some more information on the contents can be found in the beginnings of sections.

 \begin{rema} It appears that computation of associated orders in the cases where no freeness results are known to hold (cf. Remark \ref{rord}(2))  are rather rare. The only example of this sort known to the authors is the (main) Theorem 2.4 of \cite{bylt}.
 \end{rema}



\section{Some notation and basic 
 associated Galois module theory}\label{sbase} 

In this section we 
 recall some basics of the theory of associated Galois modules.
 
 In \S\ref{sbd} we mostly introduce notation.
 
 In \S\ref{skkk} we recall the theory of the isomorphism $\phi:\kkk\to \kg$. The results of this section do not differ much from the corresponding statements of \cite{bgm3}; yet the exposition is new. We will not really need 
 the results of this section till \S\ref{smain}.

\subsection{Basic notation and definitions}\label{sbd} 

$\kk$ is a totally ramified Galois extension of complete discrete valuation fields, $G$ is its Galois group, and $n$ is its degree.

We will write $\gmo\subset \go$ (resp.  $\gm\subset \gok$)  for the maximal ideal and the rings of integers of $k$ and $K$, respectively.

$\ovk=\go/\gmo\cong\gok/\gm$ is the residue field both of $k$ and $K$. 
 Denote by $r$ the canonical epimorphism $\gok\to \ovk$.

 $p>0$ is the characteristic of $\ovk$, $v$ is the discrete valuation on $K$.

We set $d=v(\gd^{\kk})-n+1$ be the  {\it ramification depth} of $\kk$; here $\gd^{\kk}$ is the different of this extension.

 The characteristic of $K$ may be either $p$ or $0$.
 
 $\tr_{E/L}$ is the trace operator corresponding to a finite extension $E/L$ (of complete discrete valuation fields); we will mostly need $\tr=\tr_{\kk}$.
 
 $\pi$ will be a fixed uniformizing element of $K$. 

 \begin{lem}\label{ltr}
 $r\circ \tr$ yields a $\ovk$-linear isomorphism $\gm^{-d}/\gm^{1-d}\to \ovk$.
\end{lem}
 \begin{proof}
 Essentially by the definition of the different,  $\tr(\gm^{1-d})= \gmo$ and  $\tr(\gm^{-d})=\go$. Hence we get a non-zero map which is obviously $\ovk$-linear. Since the multiplication by $\pi^{-d}$ gives an isomorphism $\gm^{-d}/\gm^{1-d}\to \ovk$,  $r\circ \tr$ is an isomorphism  indeed.
\end{proof}

Consequently, the element $\cpi=\tr(\pi^{-d})$ belongs to $ \go^*$.

Now let us define some associated Galois modules; we will define more of them in \S\ref{smore} below. 

 For $i\in \z$ we set $$\gc_i=\{f\in K[G]:\
\min_{x\in K^*}(v(f(x))-v(x))\ge i\};\ \ga_i=\gc_i\cap k[G];$$
here 
we take the obvious action of the set $K[G]$  on $K$.\footnote{Note that the usual group algebra multiplication in $K[G]$ does not correspond to the composition of endomorphisms of $K$. Yet this problem does not occur if one multiplies elements of $k[G]$ only; thus $k[G]$ acts on $K$ as a ring.}
Obviously, $\gc_i$  (resp. $\ga_i$) give a separated exhaustive filtration on $K[G]$ (resp. $k[G]$).

Clearly, for $\si\in G$ we have $\si-1\in \ga_i$ if and only if $i\ge h$, where $h$ is the  (lower) ramification jump for $\si$ (see Definition II.4.5 of \cite{fev}); 
 cf. the proof of Theorem \ref{tcomp}(1) below.

We will write $d(f)=i$ whenever $f\in K[G]\setminus\ns$ and $f\in \gc_{i+d}\setminus  \gc_{i+d+1}$; we will justify this shift by $d$ (along with the $\rcpi$ multiplier below) in \S\ref{skkk}. Obviously, this definition gives a well-defined function $\kg\setminus \ns\to \z$.

Let us also describe functions corresponding to  factors of the filtration $\gc_i$. 

For $i\in \z$ and  $f\in \gc_{i}$ we set \begin{equation}\label{epi}
p_i(f)=\rcpi \sum_{j=0}^{n-1}r(f(\pi^{j-i})/\pi^j) X^j\in \rr= \ovk[X]/(X^n-1).   
\end{equation}

We will write $p_i(f)\sim g\in \rr$ if there exists $c\in \ovk\setminus \ns$ such that $p_i(f)=cg$.

Moreover,  if $f\in \gc_{i} \setminus  \gc_{i+1}$ 
 then we set $\gp(f)=p_i(f)$.  Thus $\gp$ gives a well-defined function  $\kg\setminus \ns\to \rr$. 

\begin{lem}\label{lpi}
$p_i$ gives a $\ovk$-linear isomorphism $\gc_{i} /  \gc_{i+1}\to \rr$ 
(of $\ovk$-vector spaces).
\end{lem}
 \begin{proof}
Immediate from our definitions. 
\end{proof}


\subsection{Definitions and statements related to $\kkk$} \label{skkk} 

The algebra $\kkk$ is really important for our paper. Some of its properties extend to more general field extensions, and we will start from statements of this type. 

For this purpose we introduce the following notation.

We will always assume below that $\kp$ is a totally ramified extension of complete discrete valuation fields,\footnote{In several statements we don't need $\kp$ to be Galois. Note however that  we don't really need non-Galois extensions in this paper. }  $n'$ is its degree, $\gop$ (resp. $\gopo$) is the ring of integers of $K'$ (resp. $k'$),  $\gmp$ is the valuation ideal of $K'$,  $\ovkp$ is the residue field of $K'$ and $k'$, and $\pip\in \gmp$ is an uniformizing element.

For $i\in \z$ we set $X_i^{\kkp}=\sum_{j\in \z} \gmp^j\otimes \gmp^{i-j}\subset \kkkp$. 

Below we will omit the  upper index $\kkp$ in all the notation in the case $\kp=\kk$. We write just $t$ for the automorphism of $\kkp$ (or of $\kkk$) that 
  swaps the factors of the tensor square.
 
 By default, all tensor products below are that over the ring $\go$. Yet when we will write $\otimes$  when treating subsets and elements of $\kkp$ then we will assume that this symbol means $\otimes_{\gopo}$.

\begin{pr}\label{pkkp}
Assume $i,j\in\z$.

1. $X^{\kkp}_i \cdot X^{\kkp}_j\subset X^{\kkp}_{i+j}$. 

Consequently, $X^{\kkp}_0$ is a subring of $\kkp$ 
 and $X^{\kkp}_i$ is an $X^{\kkp}_0$-module.
 

2. Moreover, 
  $X^{\kkp}_i/X^{\kkp}_{i+1}$ is a one-dimensional free module over the ring $X^{\kkp}_0/X^{\kkp}_{1}$. 
  
 3.  $X^{\kkp}_i=\bigoplus_{l=0}^{n'-1}\pip^l\otimes \gmp^{i-l}$.
  
 4. For any $\eps\in \gop^*$ we have $\eps\ot \eps\ob\in 1+X^{\kkp}_{1}$.
 
 Consequently, for $x\in K'{}^*$ and $\eps_1,\eps_2\in \gop^*$ the class of $\eps_1x\otimes \eps_2 x\ob$ in  $X^{\kkp}_0/X^{\kkp}_{1}$ equals that of $\eps_1\eps_2(\pip\otimes \pip{}\ob)^{v'(x)}$, where $v'$ is the discrete valuation on $K'$, and the class of $x\otimes  x\ob$  is $1$ if  $n'\mid v'(x)$. 
 
 5. There exist a unique isomorphism $$\rxp: X^{\kkp}_0/X^{\kkp}_{1}\to  \rrp=\ovkp[X]/(X^{n'}-1)$$ of $\ovkp$-algebras with a unit that sends $\pip\otimes \pip{}\ob$ into $X$.
 Moreover, this element of $X^{\kkp}_0/X^{\kkp}_{1}$ along with the aforementioned isomorphism do no not depend on the choice of $\pip$.
 
 6. $t$ is a ring automorphism that restricts to $X^{\kkp}_i$.
  
  Moreover, for any $\al\in X^{\kkp}_0$ we have $\rxp(t(\al))=  \rxp(\al)(X\ob)$; 
   note here that $X$ is invertible in $\rrp$. 
 \end{pr}

\begin{proof}.
1. Obvious.

2. The previous assertion implies that $X^{\kkp}_i/X^{\kkp}_{i+1}$ is a module over $X^{\kkp}_0/X^{\kkp}_{1}$ indeed. It remains to note that the multiplications by $1\otimes \pip^{i}$ and $1\otimes \pip^{-i}$ give mutually inverse $X^{\kkp}_0/X^{\kkp}_{1}$-module isomorphisms between $X^{\kkp}_0/X^{\kkp}_{1}$ and  $X^{\kkp}_i/X^{\kkp}_{i+1}$.

3. Since 
$\{\pip^{l+a},\ 0\le a\le n'-1\}$, give a $\gopo$-base of $\gmp^{l}$ for any $l\in \z$, and $\gmp\subset K'$,  $\bigoplus_{0\le a\le n'-1} \pip^{l+a}\otimes_{\gopo} \gmp^{i-l}= \gmp^l\otimes_{\gopo} \gmp^{i-l}\subset \kkp$. 
 Summing up these equalities for 
all $l,\ 0\le l\le n'-1$, we easily deduce the statement in question.

4. Firstly, $\kp$ is totally ramified; hence $\eps$ can be presented as $\eps'(1+\de)$, where $\eps\in k'$ and $\de\in \gmp$. Hence $\eps\otimes \eps\ob=\eps'\otimes \eps'{}\ob\cdot (1+\de)\otimes (1+\de)\ob=(1+\de)\otimes (1+\de)\ob\in 1+X^{\kkp}_{1}$ indeed.
This immediately implies 
$\eps_1(x\otimes \eps_2 x\ob- \eps_2 (\pip\otimes \pip{}\ob)^{v'(x)})\in X^{\kkp}_{1}$

Lastly, if 
 $v'(x)$ is divisible by $n'$ then $x=x'\eps_x$ for some $x'\in k'$ and $\eps_x\in \gop^*$; hence the images of  $x\otimes x\ob$ and $x'\otimes x'{}\ob=1$ in $X^{\kkp}_0/X^{\kkp}_{1}$ coincide.

5. 
Assertion 3 easily implies that $X^{\kkp}_0/X^{\kkp}_{1}$ is generated by the classes of  $\pip^l\otimes \pip^{-l}$, $l\ge 0$,  as an $\ovkp$-vector space, and 
   the classes of $1,  \pip\otimes \pip^{-1},\dots, \pip^{n'-1}\otimes \pip^{1-n'}$ are $\ovkp$-independent in  it. Hence to construct the isomorphism in question it suffices to verify that   $\pip^{n}\otimes \pip^{-n}\in 1+X^{\kkp}_{1}$. Now, 
    this statement is given by the previous assertion, that also immediately implies the independence statments in our assertion.
    
    6. All the statements in question except the last one are obvious.
    
    The 
     equality is easy as well. We clearly can present $\al$ as $\sum_{0\le l\le n'-1} a_{l} \pip^l\otimes \pip^{-l}+\al'$ for some $a_{l}\in \gopo$   and $\al'\in X^{\kkp}_1$; see 
     assertions 3 and 5. Hence it suffices to verify the statement for $\al=\pip^l\otimes \pip^{-l}$, and in this case it is obvious.
       
 \end{proof}

Now we return to the extension $\kk$ and relate $\kkk$ to $K[G]$. We recall a definition that is important for our arguments. 

\begin{theo}\label{tkkk}
Assume $i,j\in\z$.

1. The $k$-vector space homomorphism  $\kkk\to K[G]$ that sends $x\otimes y$ (for $x,y\in K$) into  $x\sum_{\si\in G}\si(y)\si$, is bijective.\footnote{And $\phi$ is also $K$-linear if we multiply elements on $K$ by the first component in $\kkk$.}

2.  
 For any $ \sum x_i\otimes y_i\in \kkk$ and  $z\in K$ we have
$\phi(\sum x_i\otimes y_i)(z)=\sum_i x_i \tr(y_i z)$.

3. For any $\al,\be\in \kkk$ we have $\phi(\al)*\phi(\be)=\phi(\al\be)$, where we set $$\sum_{\si\in G} a_{\si}\si *\sum_{\si\in G} b_{\si}\si=\sum_{\si\in G} a_{\si}b_{\si}\si.$$
  
 4. $\gc_{i+d}=\phi(X_{i})$. 
 
 5. 
 For any $i\in \z$ the following two functions on $X_i$ coincide: $\rx_i=x\mapsto \rx(x\cdot (1\otimes \pi^{-i}))$ and  $p_{d+i}\circ \phi$.

6. $\phi(t(\al))=\sum_{\si\in G}\si(a_{\si\ob})\si $, where $\phi(\al)=\sum_{\si\in G}a_{\si}\si $.

\end{theo}

 \begin{proof}
  Assertions 
   1---4 are given by Lemma 1.1.1, Proposition 1.3.1, and Proposition 2.4.4 of \cite{bgm3}, respectively.
 
 5. Both of these functions are clearly additive and annihilate $X_{d+i+1}$. Moreover, they are obviously $\ovk$-linear if considered as functions from $X_{d+i}/X_{d+i+1}$. Hence
  it suffices to compare their values on $\pi^j\otimes \pi^{i-j}$ for 
  $0\le  j\le n-1$. 
   Now,  $\rx(\pi^j\otimes \pi^{-j})=X^j$; see Proposition \ref{pkkp}(4). Next, 
    $$p_{d+i}( \phi (\pi^j\otimes \pi^{i-j}))=     \rcpi \sum_{l=0}^{n-1}r(\pi^j\tr(\pi^{l-j-d}) \pi^{-l}) X^l.$$ Applying Lemma \ref{ltr} we easily obtain  $p_{d+i}( \phi (\pi^j\otimes \pi^{i-j}))=X^j$, and this concludes the proof.
    
    6. Obviously, it suffices to verify this equality for $\al =x\otimes y$, where $x,y\in K$. Now, $\phi(y\otimes x)=y\sum_{\si\in G} \si(x)\si=\sum_{\si\in G} \si(x\si\ob(y))\si =\sum_{\si\in G}\si(a_{\si\ob})\si$ indeed.


\end{proof}

\begin{rema}\label{rdp}
1. In \cite[Definition 2.8.1]{bgm3} $d_i(f)$ 
for $f\in \gc_{i+d}$  was essentially defined as 
 $\rx_i$ (see Theorem \ref{tkkk}(5)). Thus we have just checked the compatibility of the two definitions.

2. Now we 
 will study the relations between distinct extensions and the corresponding $\kkp$'s. 
  We are mainly interested in the case of Galois extensions, and it will be convenient for us to 
denote the bigger extension by $\kk$. Note hower, that Proposition \ref{pkkk}(1,2) below 
  is clearly valid without assuming that the extension $\kk$ is Galois.



\end{rema}

\begin{pr}\label{pkkk}
Assume 
$k'\subset K'\subset K$, $k' \subset k \subset K$,
$k/k'$ is a finite extension, and $\al'
\in X_i^{\kkp}$ for some $i\in \z$. 

Denote the ramification index of $K/K'$ by $e$ and the image of $\al'$ in $\kkk$ by $\al$.


1. Assume that $\al'= \sum_{0\le l\le n'-1} a_{l}\pip^j\otimes \pip^{i-l}+\be$ for some $\be\in X_{i+1}^{\kkp}$ and  $a_{ij}\in \gopo$ (cf. Proposition \ref{pkkp}(3)). Then $\al\in X_{ie}$ 
and $\rx_{ie}(\al)=c^i\sum r(a_{l})X^{le}$, where $c=r(\pip/\pi^{e})$. 


2. Assume   
$p\nmid [k:k']$, 
    $n$ is a power of $p$, and $K=K'k$. Then $n=n'$ and  $\al' \in X_{i+1}^{\kkp}$ if and only if $\al \in  X_{(i+1)e}$.

3. Assume that  $k=k'$ and $K'/k$ is a Galois subextension of  $\kk$. Then $\phi(\al)=\phi^{\kkp}(\al')\circ \tr_{K/K'}$.

\end{pr}
\begin{proof}
1. Once again, it clearly suffices to verify this statement in the case $\al'=\pip^l\otimes \pip^a$ for $0\le l<n$, $a\ge i-l$. Now, if $a>i-l$ then we obviously have $\al\in X_{(i+1)e}$; thus we can assume $a=i-l$. In this case $\al\in X_{ie}$ and $\rx_{ie}(\al)=\rx(\pip^l\otimes (\pip^{i-l}/\pi^{ei}))=
  \rx((\pip/\pi^{e})^{i}\otimes 1 \cdot \pi^{le}\otimes \pi^{-le}\cdot (\pip/\pi^{e})^{l-i}\otimes  (\pip/\pi^{e})^{i-l} 
=c^iX^{le}$ indeed; see Proposition \ref{pkkp}(4).

2. Since the degrees of $\kp$ and $k/k'$ are coprime, these extensions are linearly disjoint. Hence $n=n'$ indeed.

To verify the equivalence in question is suffices to apply the previous assertion and note that the corresponding homomorphism $X^{\kkp}_0/X^{\kkp}_{1}\to X_0/X_{1}$ is injective. 

3. Clearly, it suffices to verify that for any $z\in K$ we have  $\phi(\al')(z)=\phi^{\kkkp} (\al') (\tr_{K/K'}z).$

Now if $\al'=\sum x_i\otimes y_i$ for $x_i,y_i\in K'$ then (by Theorem \ref{tkkk}(2)), $$\phi(\al')(z)=\sum x_i \tr(y_iz)=\sum x_i \tr_{K/K'}(y_i \tr_{K/K'}z)=\phi^{\kkkp}(\al') (\tr_{K/K'}z)$$ indeed.

\end{proof}

\section{Main associated Galois modules calculations}\label{scalc}

This section is devoted to constructing $\go$-bases of the modules $\ga_i$.

In \S\ref{sgrad} we introduce some new and rather simple theory for constructing bases of this sort; we call the corresponding notions {\it graded-independent sets} and {\it graded bases}.

In \S\ref{scomp} we study the graded independence and the values of the functions $d$ and $\gp$ for elements of the form $\prod_{i=1}^a(\si_i-1)$ for $a\le p-1$.

In \S\ref{smain} we apply all the earlier theory to the study extensions with Galois group $(\zpz)^2$. We are able to construct a ("simple")  graded base in the case where the ramification jumps of $\kk$ are distinct modulo $p^2$; 
 see Theorem \ref{tmain}.

\subsection{On graded-independent sets and bases}\label{sgrad}

Let us give some more simple definitions related to associated Galois modules.

\begin{defi}\label{dbase}
Assume $B\subset \kg\setminus\ns$. 

1. For $i\in \z$ we set $B_i=\{f\in B,\ d(f)\equiv i\mod n\}$.

2. We will say that $B$ is  {\it graded-independent}  if for any $i\in \z$ the set $\gp(B_i)\subset \rr$ is linearly independent over $\ovk$.

3. We call 
$B$ a  {\it graded base} for $\kk$
 whenever $B$ is  graded-independent and generates $\kog$ as a $k$-module (so, $B\subset \kog$).

\end{defi}

\begin{pr}\label{pbase}

Assume $B\subset \kg\setminus\ns$ is a graded-independent set.

1.  Chose a non-zero function $c:B\to k$ and set $m=\min_b (v(c_b)+d(b))$.
 
 Then $d(\sum_{b\in B} c_b b)=m$ and $\gp(\sum_b c_b b)=\sum_{b\in B_m} r(c_b)\gp(b)$.

2. For any  
 $i\in \z$ we have $\gc_i\cap(\bigoplus_{b\in B} k\cdot b) =\bigoplus_{b\in B} \pik^{[(i-d-d(b)-1)/n]+1} b\cdot\go$.  
 
   Moreover, if $B$ contains $n$ elements and $B\subset \kog$  then $B$ is a graded base for $\kk$.

\end{pr}
\begin{proof}
1. Both equalities are easy. Obviously, $d(\sum_{b\in B} c_b b)\ge m$ and   $d(\sum_{b\in B\setminus B_m} c_b b)> m$. Thus it remains to apply Lemma \ref{lpi}.

2. Follows  from assertion 1 immediately.
\end{proof}

\begin{rema}\label{rbase} Consequently, if $B$ is a graded base for $\kk$ then the restriction of the function $d$ (resp. $\gp$) to $B$ completely determines the values of these functions on $\kog\setminus \ns$. 
\end{rema}

\subsection{Simple calculations for ``short compositions'' of $(\si_i-1)$}\label{scomp}

Starting from this moment we will always assume that $n$ is a power of $p$.

First we recall a statement on ramification jumps in this case. 

\begin{lem}\label{lcongj}[\cite[Proposition IV.11]{serre}]  

All ramification jumps of $\kk$ are congruent modulo $p$.
\end{lem}

So we set $\ovh,\ 0\le \ovh<p,$ to be the common residue of the ramification jumps of $\kk$ modulo $p$.

\begin{theo}\label{tcomp}

Let $\si_i,\ 1\le i \le a$, belong to $G$ (we do not assume $\si_i$ to be distinct). We will write $h(\si_i)$ for the ramification jumps corresponding to $\si_i$, $\sum =\sum_{i=1}^ah(\si_i)$, $\prod=\prod_{i=1}^a(\si_i-1)\in \kog$.

Then the following statements are valid.
 
1. $\prod\in \ga_{\sum}$ and $p_{\sum}(\prod)\sim \sum_{j=0}^{n-1} (\prod_{l=1}^{a} (j-l\ovh) 
)  X^j$ (see (\ref{epi})). 

2. Assume in addition that $a<p$. Then $(X-1)^{n-a-1}\mid   p_{\sum}(\prod)$ and  $(X-1)^{n-a}\nmid   p_{\sum}(\prod)$; hence $d(\prod)=\sum -d$.

Moreover, if $\ovh\neq 0$ then   $p_{\sum}(\prod)\sim (X^{\ovh}-1)^{n-a-1}$.

3. 
Assume that  $p^a\le n$, $\ovh\neq 0$, 
and for any $i,\  1\le i \le a-1$ we have $h(\si_{i+1})-h(\si_i)=p^is_i$, where $s_i\in \z\setminus p\z$. 
Set $B$ to be the set of   $\prod (\si_i-1)^{n_i}$ for 
 $(n_i)$ running through all sets of non-negative integers such that $\sum n_i<p$.

Then for any $s\in \z$   the corresponding set $B_s$ consists of at most one element. 
Consequently, 
 $B$ is graded-indepent. 
\end{theo}
\begin{proof}
1. Assume $(\si_i-1)(\pi)=c_i\pi^{h(\si_i)+1}\eps_i$ for some $c_i\in \go^*$ and $\eps_i\in 1+\gm$. Then for any $j\in \z$ we obviously have $(\si_i-1)(\pi^j)=c_i^j\pi^{h(\si_i)+j}\eps_{ij}$ for some  $\eps_{ij}\in 1+\gm$; recall here that $h(\si_i)>0$ since $\kk$ is wildly ramified.
 Combining this statements for all powers of $\pi$ and all $\si_i$ we easily obtain $p_{\sum}(\prod)=r(\prod c_i) \sum_{j=1}^{a} (j-l\ovh)  X^j$.  Since $r(\prod c_i)\neq 0$, we obtain the result. 

2. 
Since $X^n-1=(X-1)^n$ in $\ovk[X]$,  for $b\ge 0$ we have $(X-1)^{n-b}\mid p_{\sum}(\prod)$ if and only if $(X-1)^b p_{\sum}(\prod)=0$. 
Now, the previous assumption implies that the  coefficient of  $p_{\sum}(\prod)$ at $X^j$ is a non-zero polynomial in $r(j)$ of degree $\sum<p$. Next, 
in $R$ we have $(X-1)(\sum_{i=0}^{n-1}g(i)X^i)=\sum_{i=0}^{n-1}(g(i-1)-g(i)) X^i$. Hence the well-known formula for the $s$th difference of a polynomial   of degree $s$ gives the divisibility statements in question.

The ``moreover'' statement is just a simple calculation of coefficients. 

3. According to Proposition \ref{pbase}, if all the sets $B_s$ consist of at  most one element then $B$ is graded-independent indeed.

 Thus for two sets of non-negative integers $(n_i)$ and $(n_i')$ whose sums do not exceed $p-1$ it suffices to verify that $\sum n_ih(\si_i)\equiv \sum n'_ih(\si_i)\mod n$ implies  $(n_i)=(n_i')$ . 

Now, for any numbers $o_i$ we have 
 $\sum_{i=1}^a  o_ih(\si_i)=
q_1h(\si_1)+ \sum_{j=1}^{a-1}q_{j+1}p^{j}s_j$, where $q_l=\sum_{r=l}^ao_i$ for any $l,\ 1\le l\le a$. Now we take $o_i=n'_i-n_i$; then 
for any $j\ge 0$ clearly 
 $q_j$ is an integer that is zero if it is divisible by $p$. Applying obvious induction we obtain that all $q_j$ vanish; hence  $(n_i)=(n_i')$ and we obtain a contradiction.
\end{proof}

\begin{rema}\label{rmax}
The case $\ovh= 0$ is rather ``rare''.  This can only happen if $\cha K=0$  and $G$ is cyclic; see Proposition III.2.3 of \cite{fev} and  Exercises IV.2.3(c,f) of \cite{serre}.\footnote{Here we also use the well-known fact that there exists a degree $p$ subextension in $\kk$ whose ramification jump equals the smallest ramification jump of $\kk$.} 
Moreover (see Proposition III.2.3 of \cite{fev} once again), if $G=\zpz$ and   $\ovh= 0$ then for any other Galois extension $K'/k$ with Galois group $ \zpz$ the ramification jump in it is at most $p$.



\end{rema}




\subsection
{Calculations in the case $G=(\zpz)^2$} 
\label{smain}

Now we pass to 
 to more specific statements. We assume that $K=K_1K_2$, where $K_i/k$ are degree $p$ extensions whose ramification jumps are $h_2>h_1>0$.  Thus $n=p^2$, $G=(\zpz)^2$, and  simple ramification theory calculations give the following statement.

\begin{pr}\label{pram}
Set $\si_1$ (resp. $\si_2$) to be a non-trivial element of $G$ whose restriction to $K_2$ (resp. $K_1$) is trivial. 

1. Then $p\nmid h_1$.

2.  The ramification jumps corresponding to $\si_1$ and $\si_2$ equal
$h_1$ and $\wh=ph_2-(p-1)h_1$, respectively. Moreover, these ramification jumps are prime to $p$ and  $h_1\not \equiv h_2\mod p$ if and only if $h_1\not \equiv \wh\mod p^2$.

3. $d=(p-1)(ph_2+h_1)$, and the ramification depths of $K_1/k$ and $K_2/k$ equal $(p-1)h_1$ and $(p-1)h_2$, respectively.

\end{pr}
\begin{proof}
1. 
  $h_1$ is prime to $p$ since $h_1<h_2$; see Remark \ref{rmax}.

2. The calculation of jumps in this case is really easy; see Exercise III.3.2b of \cite{fev}. 
It immediately yields the equivalence in questions.
 
 3.  We apply  the (definition of  the depth of ramification along with the)  well-known formula for the different given by Proposition 4 in \cite[\S IV.1]{serre}. 
  We immediately obtain the values of ramification depths of $K_1/k$ and $K_2/k$, and it remains to note that 
  $d=(p-1)(ph_2-(p-1)h_1) +(p^2-p)h_1 $. 
\end{proof}

Now we 
pass to the main specific theorem of this paper. We set $a=i+j$ (for  integer $ i,j,\ 0\le i,j \le p-1$) and define  the following (piecewise linear) function: 
  we set $H(i,j)=h_1i+\wh j-d$ whenever $a< p-1$ and $H(i,j)=(pi-(p-1)^2)h_1+ph_2j$ 
    otherwise. 
  
  We also define the following $P(i,j)\in \rr$ 
  for (integer) $i,j\ge 0$: we set 
  $P(i,j)=
  (X^{h_1}-1)^{n-a-1}$ if $a< p-1$.
   Next, for $i+j\ge p-1$ we set  $P(i,j)=(\sum_{s=0}^{p-1} (\prod_{l=1}^{i} (s-lh_1))X^{ps}) (\sum_{t=0}^{p-1} (\prod_{l=1}^{j} (t-lh_2))X^{pt})$. 

\begin{theo}\label{tmain}
Adopt the notation and assumptions of the previous proposition. For $ 0\le i,j\le p-1$ we set $f_{ij}=(\si_1-1)^i(\si_2-1)^j\in k[G]$, $a=i+j$.

1. Then 
for $ 0\le i,j\le p-1$ we have  $d(f_{ij})=H(i,j)$ and $\gp(f_{ij})\sim P(i,j)$.

Consequently, 
 $p\mid d(f_{ij})$ if and only if $a\ge p-1$.

2. Moreover, if $p\nmid h_2-h_1$ then 
 $B=\{f_{ij} : \ 0\le i,j\le p-1\} $ 
  is a graded base of $\kk$. 

\end{theo}
\begin{proof}
1. We start from the study of $d(f_{ij})$ and $\gp(f_{ij})$.

In the case $a<p-1$ the corresponding statements are given by Theorem \ref{tcomp}(2).

Now assume $a\ge p-1$. This calculation is the main application of the multiplication $*$ (and of the other statements related to $\kkk$) in our paper.
We note that  $f_{ij}=(\osi_1-1)^i\circ \tr_1*(\osi_2-1)^j\circ \tr_2$, where $\osi_1$ (resp. $\osi_2$) is the restriction of $\si_1$ (resp. $\si_2$) to $K_1$ (resp. $K_2$), $tr_{1,2}$ are the trace operators from $K$ into $K_1$ and $K_2$, respectively, and $*$ is the coefficientwise multiplication on $K[G]$ introduced in Theorem \ref{tkkk}(3).
Next,  Theorem \ref{tcomp}(2) allows us to make the corresponding computations in $K_1/k$ and $K_2/k$ easily. We obtain
$d^{K_1/k}(\osi_1-1)^i=(i-p+1)h_1$ (see Proposition \ref{pram}(3)),  $\gp^{K_1/k}(\osi_1-1)^i\sim \sum_{s=0}^{p-1} (\prod_{l=1}^{i} (s-lh_1))X^{s}$, $d^{K_2/k}(\osi_2-1)^j=(j-p+1)h_2$, $\gp^{K_2/k}(\osi_2-1)^j\sim \sum_{s=0}^{p-1} (\prod_{l=1}^{j} (s-lh_2))X^{s}$. Hence Theorem \ref{tkkk}(3,4) along with Proposition \ref{pkkk}(1,3) imply that $f_{ij}\in \ga_{H(i,j)}$ and $p_{H(i,j)+d}(f_{ij})\sim P(i,j)$.
$\;$It remains to note that $ P(i,j)\neq 0$ (in $R$) since it is not divisible by $(X-1)^n$; see  Theorem \ref{tcomp}(2).

Lastly, if $a<p-1$ then $H(i,j)=h_1i+\wh j-d\equiv h_1a-d\equiv h_1(p-1-d)\not \equiv 0\mod p$; see Proposition \ref{pram}.   It remains to note that $p$ obviously divides $  H(i,j)$ if $a\ge p-1$.

2. Firstly, $f_{ij}$ obviously give a $k$-base of $\kog$. So, it remains to verify that for any $s\in \z$  the set $\gp(B_s)\subset \rr$ is linearly independent over $\ovk$, where $B_s=\{f\in B : \ d(f)\equiv s\mod n\}$.

We consider two cases.

First assume $p\nmid s$. By assertion 1 the corresponding $B_s$ is contained inside the set  $B=\{f_{ij} : \ 0\le i,j,\ i+j<p-1\} $; hence 
 the set $B_s$ consists of at most one element according to
Theorem \ref{tcomp}(3).

Now we pass to the case $p\mid s$. The argument is similar to the proof of Theorem \ref{tcomp}(3). According to assertion 1, $B_s= \{f_{ij} : \ 0\le i,j\le p-1,\ i+j\ge p-1,\ (pi-(p-1)^2)h_1+ph_2j\equiv s\mod n\} $. 
Hence if $(i_1,j_1)$ and $(i_2,j_2)$ are distinct elements of $B_s$ then $s-p(h_2 -h_1)j_1\not \equiv s-p(h_2 -h_1)j_2\mod n$. Consequently, $ph_1(i_1+j_1)\not \equiv ph_1(i_2+j_2)\mod s$. Since $p\nmid h_1$, it follows that for distinct elements of $B_s$ the corresponding values of $i+j$ are distinct; hence $P(i,j)$ are divisible by distinct powers of $X^p-1$ (see assertion 1). This immediately implies the $\ovk$-independence in question.

\end{proof}

\begin{rema}\label{rmain}
1. It is easily seen that the argument used for the computing of $d(f_{ij})$ and $\gp(f_{ij})$ for $i+j\ge p-1$ can be vastly generalized. Indeed, assume that totally ramified field extensions $K_s/k,\ 1\le s\le m$ (for $m>1$), are linearly disjoint, that is, $K= K_1\otimes_k K_2\otimes_k \dots \otimes_k K_m$ is a field (and their composite).
Then the same argument as above implies for any $\al_s\in X_{c_s}^{K_s\otimes_k K_s}$ we have $\prod \al_s\in X_{c}$ and $\rx_{c}(\prod \al_s)=g$, where $c=n\sum c_s/n_s$, $n_s=[K_s:k]$, and $g=\prod r_{c_s}^{K_s\otimes_k K_s}(\al_s)(X^{n/n_s})$. Moreover, one can easily express $\phi(\prod \al_s)$ in terms of $\phi^{K_s/k}(\al_s)$.

2. The main question here is whether this $g$ is not zero. Moreover, one would certainly like a large collection of elements of this sort to be graded-independent and their images with respect to $\phi$ to belong to $k[G]$. We note that $n/n_s\mid d^{\kk}(\phi(\al_s))$ and $\gp^{\kk}(\phi(\al_s))$ is a polynomial in $X^{n/n_s}$.
Thus, if $n_s$ are powers of $p$ then one doesn't have much chance to obtain a big graded-independent set if $m>2$.

Now let us describe a setting 
 generalizing  that considered in Theorem \ref{tmain}(2). We assume $m=2$,  $n_1\ge n_2$, $\al_s\in X_{c_s}^{K_s\otimes_k K_s}$ for $c_1\not\equiv c_2\mod p$ and both $\prod r_{c_s}^{K_s\otimes_k K_s}(\al_s)$ are divisible by $(X-1)$ but not divisible by $(X-1)^2$. Then a straightforward generalization of the arguments used in the proof of Theorem \ref{tmain}(2) yields that  the set $\phi(\al_1^{i}\cdot \al_2^{j})$ is graded-independent if $(i,j)$ runs through non-negative integers such that $n_2i+n_1j< n$. This set consists of  $(n+n_1)/2$ elements. Unfortunately, the authors have no idea how to complete these elements to a graded-independent base (unless $n_1=p$).
 
 Recall that 
  a vast class of extensions $K_s/k$ that contain $\al_s$ satisfying the properties in question 
   was introduced in \cite[\S3]{bgm3}; these extensions were called {\it semistable} extensions. Moreover, Theorem 3.5 and Proposition 3.4.1 of ibid. 
    can be used to construct semistable extensions explicitly (for $p\nmid c_s$)
 
 Alternatively, if $\cha k=0$, $k$ contains 
 all roots of $1$ of degree $p^r$ ($r>0$), and $K_1=K(\pi_1)$, where $\pi_1= \sqrt[p^r]{\pik}$ (cf. Remark \ref{rmax}) then for $\al^1=\pi_1\otimes \pi_1^{-1}-1$ one can easily check that $\phi^{K_1/k}(\al^1)$ belongs to $k[G]$ as well, $\al^1\in  X_{0}^{K_1\otimes_k K_1}$, and $r_0^{K_1\otimes_k K_1}(\al^1)=X-1$. 

 3. Note that the elements $\phi(\al_1^{i}\cdot \al_2^{j})$ can be expressed as products of $\phi(\al_1)$ and  $\phi(\al_2)$ with respect to the multiplication $*$; see Theorem \ref{tkkk}(3). It appears that one can replace the elements $f_{ij}$ in Theorem \ref{tmain} by elements that can be computed similarly (for certain $\phi(\al_1)$ and $\phi(\al_2)$ of a rather simple form). Yet we do not use bases of this form in this paper since 
  they are not compatible with (the general) Theorem \ref{tcomp}.
  
\end{rema}






\section{
 Some more associated modules and orders}\label{smore}
 
 Throughout this section we will assume that $k_0\subset k$,  $k/k_0$ is a totally ramified extension of complete discrete valuation fields (yet cf. \ref{rgen} below).
 
 In \S\ref{srel} we develop the theory of the modules  $\gaz_i=\ga_i\cap \kzg$. This enables us to prove an analogue of Theorem \ref{tmain}(2) in the case where the two ramification jumps are distinct but congruent modulo $p$.
 
 In \S\ref{ssmore} we consider various associated orders and some more associated Galois modules (we define $\gaz(i,j)=\{f\in \kz[G]:\ f(\gm^i)\subset \gm^j \}$). We define and study those graded bases that are convenient for constructing bases of modules of this type; we call them {\it diagonal bases}. Moreover, we prove that a graded base becomes a diagonal base if we lift the original extension by a tamely ramified extension of degree that is at least $n-1$.

 \subsection{On relative associated Galois modules}\label{srel}

  We set $e_0=[K:k_0]$, $\goz=\go\cap \kz$ is the ring of integers of $\kz$, $\piz\in \goz$ is an uniformizing element.

 Now we define associated Galois modules inside $\kzg$. 


\begin{defi}\label{dbaser}
Assume $i\in\z$, $B\subset \kg\setminus\ns$. 

1. We set $\gaz_i=\gc_i\cap \kzg$.

2. For $i\in \z$ we set $B_i^0=B^i_0(K/k_0)=\{f\in B : \ d(f)\equiv i\mod e_0\}$.

3. We will say that $B$ is  {\it $\kz$-graded independent}  if for any $s\in \z$ the set $\gp(B_s^0)\subset \rr$ is linearly independent over $\ovk$.

4. We call 
$B$ a  {\it graded base} for $(\kk,\kz)$
 whenever $B$ is  $\kz$-graded independent and generates $\kzg$ as a $k_0$-module (so, $B\subset \kzg$).

\end{defi}

\begin{pr}\label{pbaser}

I. Assume $B\subset \kg\setminus\ns$ is a $\kz$-graded independent set.

1.  Chose a non-zero function $c:B\to \kz$ and set $m=\min_b (v(c_b)+d(b))$.
 
 Then $d(\sum_{b\in B} c_b b)=m$ and $\gp(\sum_b c_b b)=\sum_{b\in B_m} r(c_b)\gp(b)$.

2. For any  
 $i\in \z$ we have $\gc_i\cap(\bigoplus_{b\in B} \kz\cdot b) =\bigoplus_{b\in B} \piz^{[(i-d-d(b)-1)/n]+1} b\cdot\goz$.    

 II. Assume that $L$ is a $\kz$-vector subspace of $\kg$. 
 
 Then there exists a $\kz$-graded independent set $B\subset L$ such that $L=\bigoplus_{b\in B_m} \kz\cdot b $.

\end{pr}
\begin{proof}
I. The proof is easy; it suffices to generalize the proof of Proposition \ref{pbase} in the obvious way.

II. We take $B$ to be the union $B^0_s,\ 0\le s<e_0$. 
 Here we set each $B^0_s$ to be a lift to $L$ of any $\ovk$-base of $(L\cap \gc_{s})/  (L\cap \gc_{s+1})$.
 
 Obviously, $B$ is $k_0$-independent. The number of elements in it clearly equals the length of the $\ovk$-module $L\cap \gc_{0}/L\cap \gc_{e_0}=(L\cap \gc_{0})/\piz (L\cap \gc_{0})$. Hence it also equals the $\kz$-dimension of $L$, and we obtain the result in question.
\end{proof}

Now we are able to generalize Theorem \ref{tcomp}(3) and extend Theorem \ref{tmain}(2).

Once again we assume that $n$ is a power $p$. We will use the notation $\vp$ of the $p$-adic valuation of integers; $w=\vp(e_0)$. 

\begin{theo}\label{trel}

I. In each of the following two cases and for any $s\in \z$ the corresponding set $B_s^0\subset \z[G]$ consists of at most one element.  

1.  $\si_i,\ 1\le i \le a$, belong to $G$, 
and for the corresponding ramification jumps we have $v_p(h(\si_{2})-h(\si_1))<v_p(h(\si_{3})-h(\si_2))<\dots< v_p(h(\si_{a})-h(\si_{a-1}))<w$.
We take $B$ to be the set of  $\prod (\si_i-1)^{n_i}$,  
 where $(n_i)$ run through all sets of non-negative integers such that $\sum n_i<p$.

2. 
The  assumptions of Theorem \ref{tmain}(1) are fulfilled and $0<v_p(h_2-h_1)<w-1$. We take $B=\{ (\si_1-1)^i(\si_2-1)^j :\ 0\le i,j\le p-1\}$.

II. Consequently, in case I(2) the set $B$ is a $\kz$-graded base of $(\kk,\kz)$.
\end{theo}
\begin{proof}
I. The proofs of both assertions are easy and similar to that of Theorem \ref{tcomp}(3). The corresponding values of $\prod (\si_i-1)^{n_i}$ were calculated in Theorem  \ref{tcomp}(2) and  \ref{tmain}(1). Now we proceed to our cases using  simple modifications of the corresponding arguments above.

1. According to Theorem  \ref{tcomp}(2) 
(and similarly to part 3 of that theorem), it suffices to verify for two sets of non-negative integers $(n_i)$ and $(n_i')$ whose sums do not exceed $p-1$  that if  $\sum n_ih(\si_i)\equiv \sum n'_ih(\si_i)\mod e_0$ then  $(n_i)=(n_i')$. 

Once again, for any numbers $o_i$ we have  $\sum_{i=1}^a  o_ih(\si_i)=
q_1h(\si_1)+ \sum_{j=1}^{a-1}q_{j+1}p^{j}s_j$, where $q_l=\sum_{r=l}^ao_i$ for any $l,\ 1\le l\le a$.
We take $o_i=n'_i-n_i$; then 
for any $j\ge 0$ clearly 
 $q_j$ is an integer that is zero if it is divisible by $p$. Applying obvious induction we obtain that all $q_j$ vanish; hence  $(n_i)=(n_i')$ and we obtain a contradiction.

2. Once again, if  $p\nmid s$ then the corresponding values of $i+j$ are less than $p-1$; see Theorem   \ref{tmain}(1). Hence the statement in question is given by the previous assertion.

Now assume $p\mid s$.  According Theorem   \ref{tmain}(1),  we should count the pairs $(0\le i<p,\ 0\le j <p)$ such that $i+j\ge p-1$ and $(pi-(p-1)^2)h_1+ph_2j\equiv s\mod e_0$. Now, if we have two pairs $(i,j)$ and  $(i',j')$ satisfying this congruence then $p^{w-1}\mid h_1(i-i')+h_2(j-j')$.
 Since $h_1\equiv h_2\mod p$, we obtain $i-i'=j-j$. Lastly, if $p^{w-1}\mid (i-i')(h_1-h_2)$ then $p\mid i-i'$; hence $i=i'$ and $j=j'$.

II. Immediate from (case 2 of) assertion I combined with Proposition \ref{pbaser}(I). 
\end{proof}

\begin{rema}\label{rgen}  It can make sense to modify Definition \ref{dbaser} to obtain more general results.  Yet the authors did not check the details here.

1. Firstly, one can try to avoid the assumption that $k/k_0$ is totally ramified. Then one should take the corresponding base field extension into account; this does not seem to be hard.

2. 
 Secondly, one can probably extend the results of this section to the case where $\kz$ does not lie in $k$ (but lies in $K$). The main difficulty here is our definition of the function $\gp$; note that $\gp(\pi^sf)=X^s\gp(f)$ for any $s\in \z$ and $f\in \kg\setminus \ns$. It appears that this problem can be avoided if one considers the function  $\gp':\ f\mapsto \gp(\phi(t(\phi\ob(f))))$ instead of $\gp$\footnote{Recall that $t$ is the automorphism of $\kkk$ that swaps the factors of this tensor square} and applies Proposition \ref{pkkp}(6) (and possibly Theorem  \ref{tkkk}(6)).

\end{rema}

\subsection{Some more associated modules (and orders) and their relation to tame lifts }\label{ssmore}

Unfortunately, the associated order $\ga(\gok_K)=\{f\in k[G]:\ f(\gk)\subset \gk \}$ does not have to be equal to any of the $\ga_i$. For this reason, we introduce some more types of associated modules. We modify slightly the notation of \cite{bgm3}. 
 We also extend it to relative modules; however, the reader may ignore this (and assume $k_0=k$ till the end of this paper).

Below $a,b,a',b',i,j$ will always denote arbitrary  integers.

 We set $\gc(i,j)=\{f\in K[G]:\ f(\gm^i)\subset \gm^j \}$; $\ga(i,j)=\gc(i,j)\cap k[G]$ and $\gaz(i,j)=\gc(i,j)\cap \kz[G]$. 

\begin{pr}\label{plij}
1. $\gc(i,j)=\phi(\gm^j\otimes \gm^{i-d-n+1})$.
 
 2. $\gc_{j-i}\subset  \gc(i,j)\subset \gc_{j-i+1-n}$.
 
 \end{pr}
 \begin{proof} 1. Recall that the following well-known statement: the different of $\kk$ equals $\gm^{d+n-1}$. Combining it with Theorem 1.2.1 of \cite{bgm3} we obtain the result.
 
 2. Obvious.
\end{proof}

\begin{rema}\label{rbaseord}
1. Consequently, $\gaz(i,j)/\gaz(j-i)$ is an $\ovk$-vector space; it obviously equals the kernel of the $\ovk$-linear map $\gaz_{j-i+1-n}/\to \gaz_{j-i} \to  \gc_{j-i+1-n}/\gc(i,j)$.  Now, a graded base $B$ for $(\kk,\kz)$ yields an explicit base of  $\gaz_{j-i+1-n}/ \gaz_{j-i}$  (see Proposition \ref{pbaser}(I.2)), and one can compute the values on this map on this base. The authors do not think that there exists any nice way for doing this in general. However, if $B$ is ``nice enough'' then the function $\gp$ is sufficient for this computation. 

Below we will essentially demonstrate this in  Propositions \ref{pgen} and Propositions \ref{pdiag}(1,3). Yet we will work in $\kkk$ instead of $\kg$; thus one has to apply the map $\phi\ob$ (along with Theorem \ref{tkkk}(4) and Proposition \ref{plij}(1)) to make the corresponding ``translation''.

2. One may call $\gaz(i,i)\subset \ga(i,i)$ the associated orders of the ideal $\gm^i$. The ``most traditional'' of them is the ring $\ga(0,0)$. 
\end{rema}

Now we study the filtration on $\kkk$ corresponding to $\gm^a\otimes \gm^b$ for $a,b\in \z$. The first simple observation here is the following one.

\begin{lem}\label{lord}


Define the following two relations on $\z^2$: $(a,b)\le^{\z^2} (a',b')$ if and only if $a\le a'$ and $b\le b'$ and  $(a,b)\sim^{\z^2} (a',b')$ if and only if   $a- a'= b'-b=nc$ for some $c\in \z$.

Then the following statements are valid.

1. $\le^{\z^2}$ is a partial order relation, and $\sim^{\z^2}$ is an equivalence relation.

2. Set $X=\z^2/\sim^{\z^2}$; we will write $[(a,b)]$ for the $\sim^{\z^2}$-equivalence class of $(a,b)$. Then $\le^{\z^2} $ induces a well-defined  partial order relation $\le^X$ on $X$; here we set $[(a,b)]\le^X [(a',b')]$ if and only if  $(a,b)\le^{\z^2} (a'',b'')$ for some $(a'',b'')\sim^{\z^2}(a',b')$.

3. $\gm^a\otimes \gm^b\supset \gm^{a'}\otimes \gm^{b}$   if and only if $[(a,b)]\le^X [(a',b')]$.

 \end{lem}
 \begin{proof} Assertions 1 and 2 are obvious. Assertion 3 is very easy as well; cf. the proof of Proposition \ref{pkkp}(4).
\end{proof}

We will write $p^X$ for the projection $\z^2\to X$.

Now we associate certain subsets of $X$ to elements and ``ideals'' of $\kkk$. 

\begin{pr}\label{pgen}

1. Then $\al$ can be presented as $\sum_{(i,j)\in G'(\al)} \eps_{ij}\cdot \pi^i\otimes\pi^j,$ where  $\eps_{ij}\in(\gok\otimes \gok)^*$ and  $G'(\al)$ is a subset of $\z^2$ such that 
the images of any two its distinct elements in $X$ are (distinct and) incomparable.

2. The set $G(\al)=p^X(G'(\al))$ is canonically determined by $\al$.

3. $\al\in \gm^a\otimes \gm^b$ if and only if $(a,b)\le^X (a',b') $ for each  $(a',b')\in G(\al) $.

\end{pr}
\begin{proof}
All these statements are rather simple and easily follow from the results of \cite[S2.4]{bgm3}; see also Proposition 2.5.2 of ibid. 
\end{proof}

Now we define sets that are ``graded independent in a strong sense''. 

\begin{defi}\label{diag}
1. We say that $\al$ is {\it diagonal} if the number $a+b$ is constant  on the set   $G(\al)=\{([a,b])\}$. 

Moreover, if this is the case then we will also say that $\phi(\al)$ is diagonal.

2. We say that a set $B\subset \kog\setminus \ns$ is {\it $\kz$-diagonal} whenever it is $\kz$-graded independent (see Definition \ref{dbaser}) and for any $s\in \z$ any non-zero $\kz$-linear combination of elements of $B_s^0$  is diagonal.


3. For a totally ramified extension $\kp$ and $\al\in \kkp$ we will write  $d^{\kkp}(\al)=i$ whenever $\al \in X_i \setminus X_{i+1}$.

\end{defi}

\begin{pr}\label{pdiag}

Assume $\al$ is diagonal and  $d^{\kkk}(\al)=i\in \z$. 

1. 
 Assume $\rx_{i}(\al)=\sum_{l=0}^{n-1}a_lX^l$ (see Theorem \ref{tkkk}(5)). Then $G(\al)=\{\ ([(l,i-l)])  |\ a_l\neq 0\}$. 
 
 2. Assume $\be\in X_{i+n-1}$. Then $G(\al)=G(\al+\be)$.

3. If $B$ is $\kz$-diagonal and $c:B\to \kz$ is a non-zero function.

Set $G'$ to be  the union of $G(\sum_{b\in B^0_s} c(b)b)$ for $s$ running through all those integers such that not all $c(b)$ are zero for $b\in B_s^0$. 

 Then $G(\sum_{b\in B} c(b)b)$ equals the set of $\le^X$-minima of $G'$.

\end{pr}
\begin{proof}
1. Immediate from our definitions.

2. The statement easily follows from the following simple observation: if $a+b\ge n-1$ then $\pi^a\ot \pi^b\in  \gok\ot \gok$.

3. Choose a representative in $\z^2$ for each $v\in G'$; denote the set of these representatives by $\tilde G'$. Then we can present $\sum_{b\in B} c(b)b=\sum_{(i,j)\in \tilde G'} \eps_{ij}\cdot\pi^i\otimes\pi^j,$ where  $\eps_{ij}\in(\gok\otimes \gok)^*$;   
Lemma \ref{lord}(3).  Moreover, if $\tilde G\subset \tilde G'$ is the subset corresponding to all  $\le^X$-minima then this lemma allows to convert this expresion into $\sum_{(i,j)\in \tilde G} \eps'_{ij}\cdot\pi^i\otimes\pi^j,$  where  $\eps'_{ij}\equiv \eps \mod (\gm\otimes \gok+\gok\otimes \gm)$. According to Proposition \ref{pdiag}(2), this yields $G(\sum_{b\in B} c(b)b)=p_X(\tilde G)$, and this concludes the proof.
\end{proof}

Now relate these notions to tame lifts.

\begin{theo}\label{tlift}
Assume 
that $K'/\koz$ is a 
totally ramified extension of complete discrete valuation fields 
 whose degree is a power of $p$, 
 $\koz\subset  k'\subset K$, $\kp$  is a Galois extension of degree $n$ with Galois group $G$,  
  $\kz/\koz$ is a 
   tamely ramified extension of degree $e$, and $B\subset K'[G]$ is a $k'_0$-graded independent set. 
We take $k=\kz k'$, $K=\kz K'$. 

1. Then $K'/\koz$ is linearly disjoint with $\kz/\koz$.

2. $B$ is also a $k_0$-independent set in $\kg$. Consequently, it is a $k_0$-graded base for $(\kk,\kz)$ if it is a $k'_0$-graded base of $(\kp,k'_0)$.

3. Assume in addition that $e\ge n-1$. Then $B$ is also $\kz$-diagonal in $\kg$.

\end{theo}
\begin{proof}
1. Obvious.

2. Immediately follows from Proposition \ref{pkkk}(1,2).

3. Clearly, it suffices to consider the case $B=B_s^0(K'/\koz)$ for some $s\in \z$. Moreover, we can clearly assume that $d^{\kkp}(b)=s$ for all $b\in B$.
Fix a non-zero  
linear combination $\al=\sum_{b\in B}c_bb$, where $c_b\in \kz$.  

Set $m$ to be the minimum of $v(c_b)$. 
Then Proposition \ref{pdiag}(2) allows us to replace $c_b$ by any $c'_b$ such that $c_b-c'_b\in \gm^{em+n}\cap \kz$. 
Consequently, we can take $c'_b= \pik^{em/n} o_b$ for some $o_b\in\goz$. Now, if $b\in B$ and  $b=b'+\sum_{0\le l \le n-1}c_{bl}\pip^l\otimes\pip^{s-l}$ for some $b'\in X_{s+1}^{\kkp}$ then $d^{\kkk}(b')\ge se+e\ge se+n-1$; see Proposition \ref{pkkk}(2). Thus it remains to verify that the element
$\al=\piz ^{em/n}\sum_{b\in B,\ 0\le l\le n-1} c_{bl}\pip^l\otimes\pip^{s-l}$ is diagonal (in $\kkk$). Now, applying Proposition \ref{pkkk}(2) one again we obtain that $G(\al)$ (is non-empty and) consists of all those $[(me+le,-le)]$ such that $v(\sum_{b\in B}c_{bl})>0$, and this concludes the proof.

\end{proof}

\begin{rema}\label{rdiabase}
1. Combining Theorem \ref{tlift} with Theorems \ref{tcomp}(3), \ref{tmain}(3) and \ref{trel} we obtain certain $\kz$-diagonal set of elements in extensions that can be obtained as tame lifts of large enough degrees. In particular, this gives an explicit algorithm for computing all associated orders whenever $G\cong (\zpz)^2 $, the ramification jumps in $K'/k'$ are distinct, and the degree 
 of $\kz/\koz$ is at least $p^2-1$; see Remark \ref{rbaseord}.
 
 2. However, the authors suspect that the elements $(\si_1-1)^i(\si_2-1)^j,\; 0\le i,j \le p-1$ give a diagonal base ``much more often''. Yet we doubt that these elements are ``optimal`` in all cases; they possibly fail to be diagonally independent if the ramification jumps are ``too large'' (and $\cha k=0$). Moreover, there probably exist totally ramified Galois extensions such that no diagonal bases exist for them (in contrast to 
  Proposition \ref{pbaser}(II)). 
 
 3. Checking that any non-zero $\kz$-linear combination of elements of $B_s^0$  is diagonal is rather simple if  $B_s^0$ consists of at most one element. However, it appears that $k$-diagonal bases satisfying this assumptions for all $s$ are rather rare. The only family of examples of this sort known to the authors is the one of {\it stable} extensions introduced in \cite[4.1]{bgm3}. This is a rather big subclass of that of semistable extensions (cf.  Remark \ref{rmain}(2)), and one can easily construct examples of extensions of this sort.
 
 Respectively, our Theorem \ref{tlift} essentially generalizes the implication (1)$\implies$(5) in Theorem 4.4 of ibid. 
\end{rema}



\end{document}